\documentclass[11pt,twoside]{amsart}
\usepackage[latin1]{inputenc}
\usepackage[T1]{fontenc}
\usepackage{mathtools}
\usepackage{graphicx}
\usepackage{tikz}
\usetikzlibrary{chains}

\usepackage[latin1]{inputenc}
\usepackage{amsmath}
\usepackage{amsthm}
\usepackage{amssymb}

\usepackage[all]{xy}
\usepackage{hyperref}
\newtheorem{thm}{Theorem}[section]

\newtheorem{prop}[thm]{Proposition}

\newtheorem{lem}[thm]{Lemma}
\newtheorem{cor}[thm]{Corollary}

\numberwithin{equation}{section}

\theoremstyle{definition}

\newtheorem{remark}[thm]{Remark}

\newcommand{\Db}{{\rm D}^{\rm b}}

\newcommand{\Aut}{{\rm Aut}}

\newcommand{\NS}{{\rm NS}}
\newcommand{\Pic}{{\rm Pic}}
\newcommand{\rk}{{\rm rk}}

\newcommand{\cal}{\mathcal}
\newcommand{\ka}{{\cal A}}

\newcommand{\kc}{{\cal C}}

\newcommand{\kl}{{\cal L}}

\newcommand{\km}{{\cal M}}
\newcommand{\ko}{{\cal O}}

\newcommand{\ks}{{\cal S}}

\newcommand{\kx}{{\cal X}}

\newcommand{\ZZ}{\mathbb{Z}}
\newcommand{\QQ}{\mathbb{Q}}
\newcommand{\RR}{\mathbb{R}}
\newcommand{\CC}{\mathbb{C}}

\newcommand{\PP}{\mathbb{P}}

\newcommand{\hh}{\mathfrak{h}}

\newcommand{\OO}{{\rm O}}
\renewcommand{\to}{\xymatrix@1@=15pt{\ar[r]&}}
\renewcommand{\rightarrow}{\xymatrix@1@=15pt{\ar[r]&}}
\renewcommand{\leftarrow}{\xymatrix@1@=15pt{&\ar[l]}}
\renewcommand{\mapsto}{\xymatrix@1@=15pt{\ar@{|->}[r]&}}
\renewcommand{\twoheadrightarrow}{\xymatrix@1@=18pt{\ar@{->>}[r]&}}
\renewcommand{\hookrightarrow}{\xymatrix@1@=15pt{\ar@{^(->}[r]&}}
\newcommand{\hook}{\xymatrix@1@=15pt{\ar@{^(->}[r]&}}
\newcommand{\congpf}{\xymatrix@1@=15pt{\ar[r]^-\sim&}}
\renewcommand{\cong}{\simeq}

\usepackage{comment}
\excludecomment{comment}
\begin{document}

\title[]{Finiteness  of polarized K3 surfaces and hyperk\"ahler manifolds}

\author[D.\ Huybrechts]{D.\ Huybrechts}

\address{Mathematisches Institut,
Universit{\"a}t Bonn, Endenicher Allee 60, 53115 Bonn, Germany}
\email{huybrech@math.uni-bonn.de}

\begin{abstract} \noindent
 In the moduli space of polarized varieties $(X,L)$ the same unpolarized variety $X$ can occur
more than once. However, for K3 surfaces, compact hyperk\"ahler manifolds,
and abelian varieties the `orbit'  of $X$, i.e.\ the subset $\{(X_i,L_i)\mid X_i\cong X\}$,
is known to be finite, which may be viewed as a consequence of the Kawamata--Morrison cone conjecture.
In this note we provide a proof of this finiteness  not relying on the cone conjecture and, in fact,
not even on the global Torelli theorem. Instead, it uses the geometry of the moduli space of polarized varieties to conclude the finiteness by means of Baily--Borel type arguments.
We also address related questions concerning finiteness  in twistor families associated with polarized K3 surfaces of CM type.

 \vspace{-2mm}
\end{abstract}

\maketitle
{\let\thefootnote\relax\footnotetext{The author is supported by the SFB/TR 45 `Periods,
Moduli Spaces and Arithmetic of Algebraic Varieties' of the DFG
(German Research Foundation) and the Hausdorff Center for Mathematics.}
\marginpar{}
}

The paper studies the connection between moduli spaces of polarized varieties, on the one hand,
and the shape of the ample cone on a fixed variety, on the other hand. To 
illustrate our point of departure, let us revue a few well-known results.

\subsection{} The classical Torelli theorem shows that two complex smooth projective curves $C$ and $C'$
are isomorphic if and only if their polarized Hodge structures are isomorphic, i.e.\ there exists a Hodge isometry $H^1(C,\ZZ)\cong H^1(C,\ZZ)$,
or, equivalently, if their principally polarized Jacobians $J(C)\cong J(C')$ are isomorphic.
Dropping the compatibility with the polarizations, so only requiring isomorphisms of unpolarized
Hodge structures or unpolarized abelian varieties, the  geometric relation
between $C$ and $C'$ becomes less clear. In moduli theoretic terms, one may wonder about the
geometric nature of the quotient map  $\km_g\to\km_g/_\sim$. Here, $\km_g$ denotes the moduli space of genus $g$ curves
and $C\sim C'$ if and only if $J(C)\cong J(C')$ unpolarized.

Similarly, two polarized K3 surfaces $(S,L)$ and $(S',L')$ are isomorphic if and only if there
exists a Hodge isometry $H^2(S,\ZZ)\cong H^2(S',\ZZ)$ that maps $L$ to $L'$.
Dropping the latter condition has a clear geometric meaning and corresponds to
considering isomorphisms between unpolarized surfaces $S$ and $S'$. Therefore, dividing out by the 
resulting equivalence relation yields a map $M_d\to M_d/_\sim$ from the moduli space of polarized
K3 surfaces $(S,L)$ of degree $d$ to the space of isomorphism classes of K3 surfaces that merely admit a polarization
of this degree. Considering only isomorphisms of Hodge structures without any further compatibilities
leads to the analogue of the aforementioned question for curves. At this time, there is no clear picture of
what the existence of an unpolarized isomorphism of Hodge structures could mean
for the geometry of the two K3 surfaces, but finiteness has recently been established in \cite{Ef}.
\smallskip

Other types of varieties, like abelian varieties, Calabi--Yau or hyperk\"ahler varieties, can be discussed from the same perspective.

\subsection{} Let us move to the cone side. For a K3 surface $S$,
the ample cone ${\rm Amp}(S)\subset\NS(S)\otimes\RR$ and its closure, the nef cone ${\rm Nef}(S)$,
are complicated and usually not rationally polyhedral. The situation changes when
the natural action of $\Aut(S)$ is taken into account. More precisely, there exists a fundamental
domain $\Pi\subset{\rm Nef}^+(S)$ for the action of $\Aut(S)$ on the effective nef cone that is rational polyhedral,
see \cite{Sterk} or \cite[Ch.\ 8]{HuyK3}. Its generalization to smooth projective varieties with trivial canonical bundle is the
cone conjecture of Kawamata \cite{Kaw} and Morrison \cite{Mor}. It has been proved for abelian varieties \cite{PS}
and for hyperk\"ahler manifolds \cite{AmVerb,AmVerb2,MarYos}. For recent progress in the case of
Calabi--Yau varieties see 
 \cite{LazPet}. 

The cone conjecture has the following somewhat less technical consequence:
Up to the action of the group of automorphisms, there exist at most
finitely many polarizations of a fixed degree, see \cite{Sterk} for the argument.
For abelian varieties  the result had been observed already in  \cite{NarNor}. 

\subsection{}
These two circles of ideas are of course linked. For example, for K3 surfaces the fibre  over $S_0$
of the map $M_d\to M_d/_\sim$ is naturally identified with the set $\{L\mid \text{ample},\ (L)^2=d\}/\Aut(S_0)$.
Due to the cone conjecture for K3 surfaces, this set is finite and, therefore, all fibres of $M_d\to M_d/_\sim$ are finite. 
Similarly, for the moduli space ${\mathcal A}_{g,d}$ of polarized abelian varieties of dimension $g$ and degree $d$ and a fixed (unpolarized) abelian variety $A_0$ there exist at most finitely many polarized abelian varieties
$(A_1,L_1), \ldots,(A_n,L_n)\in {\mathcal A}_{g,d}$ with $A_i\cong A_0$, i.e.\
the fibres of ${\mathcal A}_{g,d}\to{\mathcal A}_{g,d}/_\sim$, and consequently also of $\km_g\to\km_g/_\sim$, are finite.

It is worth emphasizing that the quotients $M_d/_\sim$, $\km_g/_\sim$, and $\ka_{g,d}/_\sim$ have no reasonable
geometric structure, which is mainly due to the fact that the fibres of the quotient maps are all finite but of unbounded
cardinality. See Section \ref{sec:unbound}, where this is discussed for K3 surfaces.

\smallskip

Note that the local Torelli theorem for these types of varieties immediately implies that the fibres of $M_d\to M_d/_\sim$,
${\mathcal A}_{g,d}\to{\mathcal A}_{g,d}/_\sim$, and $\km_g\to\km_g/_\sim$,  i.e.\ the sets
$$M_d(S_0)\coloneqq\{(S,L)\mid S\cong S_0\}\subset{M}_d, \text{~~}{\mathcal A}_{g,d}(A_0)\coloneqq
\{(A,L)\mid A_i\cong A_0\}\subset{\mathcal A}_{g,d},\text{ and }$$
$${\mathcal M}_g(C_0)\coloneqq\{C\mid {\rm Jac}(C)\cong{\rm Jac}(C_0)\}\subset{\mathcal M}_g,$$
are discrete subsets of the corresponding moduli spaces. The present paper is motivated by
the question whether the geometric nature of the three moduli spaces ${M}_d$, ${\mathcal A}_{g,d}$, and
${\mathcal M}_g$ (and others), namely being quasi-projective varieties, can alternatively be used
to deduce from their discreteness the finiteness of the three sets
$M_d(S_0)$, ${\mathcal A}_{g,d}(A_0)$, and ${\mathcal M}_{g}(C_0)$. There are well-known
instances where this naive idea indeed yields finiteness of certain discrete sets in appropriate
moduli spaces by verifying their algebraicity. As an example, we recall in Section \ref{sec:explainGT}
the proof for the finiteness of $\Aut(S,L)$ along these lines.

\subsection{} Although,  finiteness of the sets $M_d(S_0)$, $\ka_{g,d}(A_0)$,
or $\km_g(C_0)$ cannot be deduced quite so easily,
we will show that the quasi-projectivity of certain related moduli spaces can indeed be exploited.
We will demonstrate this for the moduli space  $M_d$ of compact
hyperk\"ahler manifolds of fixed degree and fixed dimension by proving the following result.

\begin{thm}\label{thm:HK}
Fix a compact hyperk\"ahler (or irreducible holomorphic symplectic) manifold $X_0$.
Let $M_d(X_0)\subset {M}_{d}$ be the set of polarized compact hyperk\"ahler manifolds $(X,L)$ of
degree $(L)^{2n}=d$ with $X\cong X_0$. Then $M_d(X_0)$ is finite.
\end{thm}

In other words, the set of ample line bundles on $X_0$ of fixed degree is finite up to the action
of the group $\Aut(X_0)$ of automorphisms of $X_0$. As the cone conjecture for hyperk\"ahler
manifolds has  recently been established in great generality in
\cite{AmVerb,AmVerb2}, see also \cite{MarYos} for a proof for the two standard series, the theorem can
also be  seen as a consequence of the cone conjecture. In fact, the full conjecture is not needed to
conclude the above result from the global Torelli theorem, a shortcut is outlined in  Section \ref{sec:proof}.
However, our approach shows that finiteness results of this type can be deduced more directly
and without using any version of the global Torelli theorem from moduli space considerations and
Baily--Borel type arguments. Alternatively, Griffiths' extension theorem can be used.  To complete the picture, we shall outline
in Section \ref{sec:KS} a proof of Theorem \ref{thm:HK} that reduces the assertion to the case of abelian
varieties via the Kuga--Satake construction.

\subsection{}
Finiteness results of this type are clearly fundamental. However, as they fail in the non-algebraic
setting, they are also quite remarkable. Recall that there indeed exist non-isotrivial families
of K3 surfaces or hyperk\"ahler manifolds with a dense subset of fibres all isomorphic to one of the
fibres. More precisely, the set of period points corresponding
to K3 surfaces or hyperk\"ahler manifolds isomorphic to a fixed one is dense except
when the complex plane $(H^{2,0}\oplus H^{0,2})$ contains a non-trivial integral class,
see \cite{Verb2,Verb3} and \cite[Rem.\ 5.7]{BL}. Theorem \ref{thm:HK} now says that this cannot
happen for polarized families. The second goal of this paper is
to prove a similar result for certain families `orthogonal' to the polarized case. 
More precisely, we study families provided by the twistor space construction. Let us restrict to K3 surfaces for simplicity and
recall that associated with any K3 surface $S_0$ endowed with a
K\"ahler class $\omega_0$, e.g.\ the one given by a polarization $L_0$,
one associates a twistor family $\ks\to\PP^1$ of K3 surfaces $\ks_t$, $t\in\PP^1$, with a natural 
K\"ahler class $\omega_t$. Only countably many of the fibres $\ks_t$ are projective and the complex
manifold $\ks$ is not even K\"ahler. We then prove the following non-algebraic
analogue of Theorem \ref{thm:HK},  see Proposition \ref{prop:finTwistor}.


\begin{thm}\label{thm:twistor} Let  $\ks\to\PP^1$ be the twistor space associated with
a polarized K3 surface $(S_0,L_0)$. Assume that $S_0$ has CM.
Then at most finitely many twistor fibres $\ks_t$ are  isomorphic to $S_0$.
\end{thm}

Despite the similarities between the two finiteness results,
they are rather different from another perspective. Namely,
in Theorem \ref{thm:HK} the number $|M_d(S_0)|$ is finite
but unbounded for varying $S_0$, whereas in Theorem \ref{thm:twistor}
the number of fibres isomorphic to $S_0$  only depends on the CM field
and can be universally bounded by $132$, cf.\ Remark \ref{rem:132}.
At this point, it is not clear whether the assumption on $S_0$ to have CM is really necessary, but the proof suggests that
it might. 

\medskip

%
%
%
%


\subsection{} We are also interested in the metric aspect. Recall that to any polarization
$L$ on a compact hyperk\"ahler manifold $X$ there is naturally associated a hyperk\"ahler
metric $g_L$ on the underlying manifold. The resulting Riemannian manifold shall be denoted
$(X,g_L)$. From the Riemannian perspective it is then natural to wonder how often the same
Riemannian manifold occurs for $(X,L)\in M_d$. Again restricting to the case of K3 surfaces
for simplicity, we prove the following result, see Corollary \ref{cor:Omet}.

\begin{thm} Let $M_d$ be the moduli space of polarized K3 surfaces $(S,L)$ of degree $(L)^2=d$.
Then the set $M_d(S_0,g_{L_0})\subset M_d$ of polarized K3 surfaces $(S,L)\in M_d$ for which
there exists an isometry $(S_0,g_{L_0})\cong (S,g_L)$ of the underlying Riemannian manifolds
is finite.
\end{thm}

\noindent
{\bf Acknowledgement:} 
Thanks to Benjamin Bakker for a discussion related to Remark \ref{rem:Lagr}, to
Bal\'azs Szendr\"oi for an email exchange concerning \cite{Szendroi}, to Klaus Hulek
for drawing my attention to \cite{La}, to Ariyan Javanpeykar for insightful comments
on Proposition \ref{prop:DMcoarse} and Section \ref{sec:proof}, and to Lisa Li for comments on
the first version.  Discussions with Fran\c{c}ois Charles and
Andrey Soldatenkov related to Section \ref{sec:twistorfini} have been very useful,
their help is most gratefully acknowledged. 
\section{Finiteness of polarizations on hyperk\"ahler manifolds}\label{sec:FinHK}
Without using the global Torellli theorem, finiteness of polarizations of fixed degree on a compact hyperk\"ahler manifold
can be  proved by applying Baily--Borel type arguments that ensure that certain arithmetic quotients and maps
between them are algebraic.

\subsection{} Consider $H^2(X_0,\ZZ)$ of a compact hyperk\"ahler (or irreducible holomorphic
symplectic) mani\-fold $X_0$ with its Beauville--Bogomolov form $q_{X_0}$ as an abstract lattice
$\Lambda$. For example, if $X_0$ is a K3 surface, then $\Lambda\cong E_8(-1)^{\oplus 2}\oplus U^{\oplus 3}$.
In general, nothing is known about  $\Lambda$ beyond the fact that it is non-degenerate of signature $(3,b_2(X_0)-3)$
and a few restrictions on $b_2(X_0)$ in low dimensions.
For $X_0$ projective, the Hodge index theorem implies that the N\'eron--Severi lattice $H^{1,1}(X_0,\ZZ)\cong{\rm NS}(X_0)$ is a
non-degenerate primitive sublattice of signature $(1,\rho(X_0)-1)$. 

\begin{remark}\label{rem:Kneser}
We shall use the following elementary facts from lattice theory, cf.\ \cite[Satz 30.2]{Kneser}, the second being a special case of the first.
Let $N$ and $\Lambda$ be arbitrary lattices.
\begin{itemize}
\item[(i)]
Up to the action of the orthogonal group $\OO(\Lambda)$, there exist at most finitely many (primitive) embeddings
$\eta_i\colon N\,\hookrightarrow \Lambda$, $i=1,\ldots,k$. In the following, we often denote
the lattices given as the  orthogonal complements by $T_i\coloneqq \eta_i(N)^\perp\subset \Lambda$.
\item[(ii)]
Up to the action of $\OO(N)$, there exist only finitely many (primitive)  classes $\ell\in N$
with fixed square $(\ell)^2=d$.
\item[(iii)]
If $\Lambda$ is definite, then there exist at most finitely many embeddings $N\,\hookrightarrow\Lambda$.
\end{itemize}
\end{remark}

\subsection{} Assume now that $N$ has signature $(1,\rho(X)-1)$ and fix an element $\ell\in N$ with $(\ell)^2=d$. Then
consider the moduli stack $\km_{(N,\ell)}$ of $(N,\ell)$-polarized hyperk\"ahler manifolds of deformation type 
(or just diffeomorphic or homeormorphic to) $X_0$. So,
$\km_{(N,\ell)}(T)$ consists of families
$\pi\colon\kx\to T$ of compact hyperk\"ahler manifolds deformation equivalent to $X_0$
 together with an embedding $\iota\colon\underline N\,\hookrightarrow R^2\pi_*\ZZ$ of locally constant systems on $T$
 which fibrewise induces
a primitive embedding of lattices $N\,\hookrightarrow \Pic(\kx_t)\cong{\rm NS}(\kx_t)\cong H^{1,1}(\kx_t,\ZZ)\subset H^2(\kx_t,\ZZ)$ mapping $\ell$ to an ample line bundle $L_t$ on $\kx_t$.

Standard moduli space constructions yield the following result.

\begin{prop}\label{prop:DMcoarse}
The stack $\km_{(N,\ell)}$ of $(N,\ell)$-polarized compact hyperk\"ahler manifolds deformation equivalent to $X_0$
 is a Deligne--Mumford stack with a quasi-polarized coarse moduli space $M_{(N,\ell)}$.
\end{prop}

\begin{proof}
We sketch the main steps in the construction. Variants of this can be found in the literature, see \cite{Beauv,Dolg}
and Section \ref{sec:paraFM} for further comments. 

Denote by $\km_d$ the moduli stack of polarized hyperk\"ahler manifolds $(X,L)$ of degree $d$. This is a
Deligne--Mumford stack
with a quasi-projective coarse moduli space, see  \cite{Viehweg} and  
\cite[Ch.\ 5]{HuyK3} for further references in the case of K3
surfaces.\footnote{The degree of a polarization $L$ would usually be given as the top intersection form
$(L)^{2m}$ or, equivalently, as $c_{\text{F}}\cdot q(L)^m$ with the Fujiki constant $c_{\text{F}}$.
As we are only looking at varieties deformation equivalent to $X_0$, the Fujiki constant
is fixed and so prescribing the Beauville--Bogomolov square $q(L)$ or the classical degree $(L)^{2m}$ amounts to the same. In particular,
it would be enough to fix the topological type of $X_0$ in our discussion.}
The map $(\kx\to T,\iota)\mapsto (\kx\to T,\iota(\ell))$ defines a morphism $f\colon\km_{(N,\ell)}\to\km_d$. As there exist
at most finitely many isometric embeddings $\ell^\perp\,\hookrightarrow \iota(\ell)^\perp\subset \Pic(X)$ (use Remark \ref{rem:Kneser} (iii)), the
morphism is quasi-finite.
In \cite{Beauv} $\km_{(N,\ell)}$ is realized as an open and closed substack of $\underline\Pic_{\kx/\km_d}^{\rho(X)}$,
where $\kx\to\km_d$ is the universal family. This is enough to conclude that $\km_{(N,\ell)}$ is a Deligne--Mumford stack
\cite[Prop.\ 4.5]{LM}.

It is not difficult to show that $f\colon\km_{(N,\ell)}\to\km_d$ is actually proper and hence finite. Therefore,
also the induced morphism between their coarse moduli spaces $M_{(N,\ell)}\to M_d$ is finite. For the existence
of the coarse moduli spaces (as algebraic spaces) one needs to use the finiteness of the stabilizers (Matsusaka--Mumford, as for K3 surfaces), see \cite{KM}. Using that $M_d$ is quasi-projective \cite{Viehweg}, yields the  quasi-projectivity of $M_{(N,\ell)}$.
\end{proof}

According to  Remark \ref{rem:Kneser},  there exist, up to the action of $\OO(N)$, at most  finitely many $\ell_1,\ldots,\ell_m\in N$ 
with $(\ell_j)^2=d$. In this sense, the quasi-projective variety
$$M_{N,d}\coloneqq M_{(N,\ell_1)}\sqcup\ldots\sqcup M_{(N,\ell_m)}$$
can be understood as the moduli space of $N$-polarized hyperk\"ahler manifolds of deformation type $X_0$ and
degree $d$.\footnote{$\ldots$ with the slight ambiguity that a primitive embedding of $N$ may send more than one
$\ell_j$ to an ample class.}

As  explained in the above proof, mapping $(\kx\to T,\iota\colon \underline N\,\hookrightarrow R^2\pi_*\ZZ)$ to $(\kx\to T,\iota(\ell))$ defines a morphism
from $\km_{(N,\ell)}$ to the moduli stack $\km_d$ of polarized compact hyperk\"ahler manifolds of deformation type $X_0$ and  degree
$d$.
This yields a morphism between their quasi-projective coarse moduli spaces
\begin{equation}\label{eqn:morph}
M_{N,d}=M_{(N,\ell_1)}\sqcup\ldots\sqcup M_{(N,\ell_m)}\to M_d.
\end{equation}
Its image consists of all the points  that correspond to polarized hyperk\"ahler manifolds
$(X,L)$ of deformation type $X_0$ for which the polarization $L$ is contained in a primitive sublattice
of the N\'eron--Severi lattice  abstractly isomorphic to $N$. In particular, 
for $N= \NS(X_0)$ the set $M_d(X_0)\coloneqq \{(X,L)\in M_d\mid X\cong X_0\}$
is contained in the image of (\ref{eqn:morph}). Clearly, if $N=\ZZ(d)$, then
(\ref{eqn:morph}) is an isomorphism $M_{N,d}\congpf M_d$.

According to Remark \ref{rem:Kneser}, the moduli space  $M_{N,d}$ also decomposes into a finite disjoint 
union $$M_{N,d}=M^1_{N,d}\sqcup \ldots \sqcup M_{N,d}^k.$$
Here, $M^i_{N,d}$ parametrizes $N$-polarized hyperk\"ahler manifolds $(X,\iota)$
for which the composition $\iota\colon
N\,\hookrightarrow H^{1,1}(X,\ZZ)\subset H^2(X,\ZZ)\cong H^2(X_0,\ZZ)\cong \Lambda$ is equivalent to
the embedding $\eta_i$. A similar decomposition exists for each fixed $\ell_j$, so
$M_{(N,\ell_j)}=\bigsqcup M^i_{(N,\ell_j)}$.\footnote{As a side remark, but very much in the spirit of our discussion,
the fact that $M_{N,d}$ is quasi-projective allows one to circumvent Remark \ref{rem:Kneser} (i), 
as it implies that there can be at most finitely many equivalence classes of embeddings $N\,\hookrightarrow \Lambda$ obtained as composition $N\,\hookrightarrow H^{1,1}(X,\ZZ)\,\hookrightarrow H^2(X,\ZZ)\cong H^2(X_0,\ZZ)\cong\Lambda$.}

\begin{remark}
The image of $M^i_{N,\ell_j}\to M_d$ defines a special cycle in the sense
of \cite{Kudla}. In particular, each $M_d(X_0)$ is contained in a finite union of
special cycles of codimension $\rho(X_0)-1$.
\end{remark} 

\subsection{} For each of the lattices $T_i=\eta_i(N)^\perp\subset\Lambda$ we consider the
usual period domain $$D_i\subset \PP(T_i\otimes\CC)$$ defined by the
closed condition $(x)^2=0$ and the open condition $(x.\bar x)>0$.
Recall that $D_i$ can be identified with the Grassmannian of positive oriented planes in $T_i\otimes \RR$ and
that it consists of two connected components, cf.\ \cite[Ch.\ 6]{HuyK3}.

The orthogonal group $\OO(T_i)$ acts naturally on $D_i$
and due to Baily--Borel \cite{BB} the quotient $\OO(T_i)\setminus D_i$ is a quasi-projective variety with finite
quotient singularities.

\begin{prop}
Associating with $(X,\iota\colon N\,\hookrightarrow \NS(X))$ the period $\varphi(H^{2,0}(X))$ yields
a well defined and algebraic map $$\pi_i\colon M^i_{N,d}\to \OO(T_i)\setminus D_i.$$
Here, $\varphi\colon H^2(X,\ZZ)\congpf \Lambda$ is any marking
for which $\varphi\circ\iota=\eta_i$.
\end{prop}

\begin{proof}
Note that $\iota\colon N\,\hookrightarrow \NS(X)$  composed with the inclusion
$\NS(X)\cong H^{1,1}(X,\ZZ)\subset H^2(X,\ZZ)$ and a marking $ H^2(X,\ZZ)\cong\Lambda$
yields a primitive embedding $\eta\colon N\,\hookrightarrow \Lambda$. By definition of $M^i_{N,d}$ this primitive
embedding is equivalent to $\eta_i$ and, hence, there indeed exists a marking  $\varphi$ with $\varphi\circ\iota=\eta_i$.
In particular, the sublattices $T_i=\eta_i(N)^\perp$  and $\varphi(\iota(N))^\perp$ of  $\Lambda$ coincide.
Hence, $\varphi(H^{2,0}(X))\subset T_i\otimes \CC$, which thus defines a point in $D_i$.
Changing $\varphi$ to $\varphi'$ still satisfying $\varphi'\circ\iota=\eta_i$ yields a period in the same
orbit of the $\OO(T_i)$-action on $D_i$. It is easy to check that isomorphic $N$-polarized
$(X,\iota)$ and $(X',\iota')$, i.e.\ both defining the same point in $M^i_{N,d}$, yield the same point in $\OO(T_i)\setminus D_i$.

Introducing markings globally over $M^i_{N,d}$ (by passing to the appropriate principal bundle) and applying local period
maps, one finds that  $\pi_i\colon M^i_{N,d}\to \OO(T_i)\setminus D_i$ is holomorphic.\footnote{This is completely analogous
to the standard arguments, see e.g.\ \cite{Beauv,Dolg,HuyK3}.}
The crucial input now is Borel's result \cite{Bor} which shows that $\pi_ i$ is automatically algebraic. Note that Borel's result only allows
quotients by torsion free groups. However, introducing finite level structures one obtains a finite cover of $M^i_{N,d}$ that then maps
holomorphically to a smooth quotient $\Gamma\setminus D_i$ for some finite index torsion free subgroup $\Gamma\subset \OO(T_i)$.
\end{proof}

\begin{cor}\label{cor:finitefibre}
The fibres of $\pi_i\colon M^i_{N,d}\to \OO(T_i)\setminus D_i$ are finite.
\end{cor}

\begin{proof}
By the local Torelli theorem, non-trivial local deformations of $(X,\eta)$ are  detected by their periods, i.e.\ by their images under
$\pi_i$. Thus, the fibres of $\pi_i$ are discrete. Now use that $\pi_i$ is algebraic, which immediately implies finiteness.
\end{proof}

\subsection{Proofs of Theorem \ref{thm:HK}}\label{sec:proof} As announced in the introduction, we shall now 
present a proof of the finiteness of $M_d(X_0)$ for hyperk\"ahler manifolds $X_0$
that avoids the Kawamata--Morisson cone conjecture and, in fact, the global Torelli theorem. We shall also
sketch an argument that uses the global Torelli theorem directly. For a third proof via the Kuga--Satake construction
see Section \ref{sec:KS}.

\medskip

-- Let $X_0$ be a compact hyperk\"ahler manifold which is assumed to be projective.
Hence, $N\coloneqq \NS(X_0)\cong H^{1,1}(X_0,\ZZ)\subset H^2(X_0,\ZZ)\cong\Lambda$
is a primitive sublattice of signature $(1,\rho(X_0)-1)$ (which by \cite{HuyHK} is known to be equivalent to the 
projectivity of $X_0$). 

The set $M_d(X_0)\subset {M}_d$
of all polarized $(X,L)$ of degree $d$ with $X\cong X_0$ is contained in the image
of the map (\ref{eqn:morph}). So, let $(X,\iota)\in M_{N,d}$ with $X\cong X_0$. Then
$\iota\colon N\congpf \NS(X)$, because an embedding of abstractly isomorphic lattices is automatically an isomorphism.
Hence, $(X,\iota)$ and $(X_0,\iota_0\colon N\stackrel{=}{\to}\NS(X_0))$ are both contained
in the same part $M^i_{N,d}\subset M_{N,d}$. Moreover, picking an isomorphism $f\colon X\congpf X_0$ yields
a Hodge isometry $f^*\colon\NS(X_0)^\perp\congpf \NS(X)^\perp$, which shows that 
$(X,\iota)$ and $(X_0,\iota_0)$ have the same image under $\pi_i\colon M^i_{N,d}\to \OO(T_i)\setminus D_i$.

Therefore, $M_d(X_0)$ is contained in the image under the map (\ref{eqn:morph}) of a fibre of $\pi_i$.
Now use Corollary \ref{cor:finitefibre}, to conclude the finiteness of $M_d(X_0)$.\qed

\begin{remark}
For Calabi--Yau threefolds a similar idea has been exploited by Szendr\H{o}i \cite{Szendroi}
and it seems plausible that his arguments can be generalized to cover Calabi--Yau manifolds
of arbitrary dimensions. Instead of Baily--Borel, so the algebraicity of the holomorphic period map,
he applies Griffiths' extension theorem \cite{CGGH} which ensures the existence of a proper holomorphic extension.

In fact, quite generally, quasi-finiteness can alternatively be deduced from the extension theorem.
See \cite{JL} for an in depth discussion of various aspects of the quasi-finiteness of  period maps and references.
\end{remark}

--  As has been mentioned, Theorem \ref{thm:HK} can also be
deduced from the global Torelli theorem which in turn relies on the existence
of Ricci-flat metrics, twistor spaces, etc. Here is a quick outline of the argument.
The reader may, for simplicity, restrict to the case of K3 surfaces.

The group of diffeomorphisms acts by a finite index subgroup on $H^2(X_0,\ZZ)$, often called the monodromy
group ${\rm Mon}(X_0)\subset \OO(H^2(X_0,\ZZ))$. For K3 surfaces the index has been determined
in \cite{Borcea} (finiteness was known to Weil). The same argument combined with \cite{HuyFin} proves finitess for arbitrary hyperk\"ahler manifolds. The subgroup respecting the Hodge structure
yields a finite index subgroup ${\rm Mon}_{\rm Hdg}(X_0)\subset\OO(\NS(X_0))$. (The finite index of the
inclusion $\NS(X_0)\oplus\NS(X_0)^\perp\subset H^2(X_0,\ZZ)$ intervenes here.) 
We know that $\OO(\NS(X_0))$ acts with finitely many orbits on the set
of all $\ell\in\NS(X_0)$ with fixed $(\ell)^2=d$, see Remark \ref{rem:Kneser}.
The last step now is to describe the image of $\Aut(X_0)\to \OO(\NS(X_0))$ as the subgroup of finite index
of all $g\in \OO(\NS(X_0))$ such that $g$ maps (at least) one K\"ahler class again to a K\"ahler class.
This last assertion is part of  the global Torelli theorem, cf.\ \cite[Ch.\ 15]{HuyK3} for the case of K3 surfaces
and \cite{Mar,Verb} for the higher-dimensional case.\qed

\begin{cor}
Let $\pi\colon(\kx,\kl)\to T$ be a polarized family of compact hyperk\"ahler manifolds (e.g.\ of polarized K3 surfaces) over a quasi-projective
base $T$. Then for any $X_0$ the set $\{t\in T\mid \kx_t\cong X_0\}$ is Zariski closed.\qed
\end{cor}

\begin{cor}
Let $N$ be a lattice of signature $(1,m)$ and $\ell\in N$ a primitive class. Then for any compact hyperk\"ahler manifold
$X_0$ there exist, up to the action of $\Aut(X_0)$, at most finitely many isometric
embeddings $N\,\hookrightarrow \NS(X_0)$ mapping $\ell$ to an ample class.
\end{cor}

\begin{proof}
The result can be proved in the spirit of the discussion above, by means of
moduli spaces of lattice polarized hyperk\"ahler manifolds and
again using Baily--Borel arguments. It can also be deduced from 
Theorem \ref{thm:HK} directly. Indeed, up to the action of $\Aut(X_0)$ there are only finitely many ample classes
$L\in\NS(X_0)$ with $(L)^2=(\ell)^2$. An isometric embedding $N\,\hookrightarrow \NS(X_0)$ mapping
$\ell$ to a fixed $L$ is then determined by the induced $\ell^\perp\,\hookrightarrow L^\perp$. As $L^\perp$ is negative
definite by the Hodge index theorem, there are only finitely many such embeddings, see Remark \ref{rem:Kneser}.
\end{proof}

\subsection{}\label{sec:unbound}
The cardinality of the `orbits' cannot be bounded. For abelian surfaces this has been observed
in \cite{Hay,La2}. This suggests a similar behavior for the associated Kummer surfaces, but
controlling the degree of the polarizations is technical.
Here, we exhibit an example of polarized K3 surfaces $(S,L)$ of degree $(L)^2=d=4d_0$ for any fixed
odd $d_0>0$ not divisible by any $p^3$ showing that $|M_{d}(S)|$ cannot be bounded.

For this, consider an increasing sequence of primes $p_i\equiv 3\ (4)$. By Siegel's theorem, the sequence
of class numbers $h_i\coloneqq h(p_i)=|{\rm Cl}(K_i)|$ of the imaginary quadratic field $K_i\coloneqq\QQ(\sqrt{-p_i})$
cannot be bounded. Then use the interpretation of ${\rm Cl}(K_i)$ as the set of ${\rm Sl}(2,\ZZ)$-equivalence
classes of binary quadratic forms $aX^2+bXY+cY^2$ of discriminant $-p_i=b^2-4ac$ or, equivalently,
as the set of isomorphism classes of oriented positive definite lattices $\Gamma$ of rank two
with intersection matrix $\left(\begin{matrix}2a&b\\b&2c\end{matrix}\right)$. It is classically known
that two forms are in the same genus if they differ by forms in ${\rm Cl}(K_i)^2$ (Gauss principal genus theorem) and
that under our assumptions in fact  ${\rm Cl}(K_i)^2={\rm Cl}(K_i)$.
cf.\  \cite[Prop.\ 3.11]{Cox}.
 Hence, for all $i$ there exist non-isomorphic
$\Gamma_{i1},\ldots,\Gamma_{ih_i}$ within the same genus.

Hence, the indefinite lattices $\tilde\Gamma_{ij}\coloneqq\Gamma_{ij}\oplus\ZZ(-d_0)$,
$j=1,\ldots,h_i$, of rank three are in the same genus
for fixed $i$. However, the genus of an indefinite  ternary form determines the isomorphism class
assuming that its discriminant is odd and indivisible by any cube,
cf.\ \cite[Ch.\ 15, Thm.\ 21]{Conway}. Hence, $\tilde\Gamma_{i1}\cong\cdots\cong\tilde\Gamma_{ih_i}$
and we shall denote this isomorphism type by $\tilde\Gamma_i$. Then the generators of $\ZZ(-d_0)$
correspond to elements $\alpha_{ij}\in\tilde\Gamma_i$. As their orthogonal complements $\alpha_{ij}^\perp$ are isomorphic to $\Gamma_{ij}$, the orbits $\OO(\tilde\Gamma_i)\cdot\alpha_{ij}\subset\tilde\Gamma_i$
are all distinct. Finally, set $N_i\coloneqq \tilde\Gamma_i(-4)$, which is lattice of
signature $(1,2)$ containing classes $\alpha_{i1},\ldots,\alpha_{ih_i}$ of square
$(\alpha_{ij})^2=d=4 d_0$ with distinct $\OO(N_i)$-orbits. By the
surjectivity of the period map, there exist K3 surfaces $S_i$ with $\NS(S_i)\cong N_i$.
As the lattices $N_i$ do not contain any $(-2)$-classes, up to sign
the $\alpha_{ij}$, $j=1,\ldots,h_i$ correspond to ample line bundles $L_{ij}$ with pairwise distinct
$\OO(\NS(S_i))$-orbits and, hence, pairwise distinct $\Aut(S_i)$-orbits. In other words,
$(S_i,L_{ij})\in M_{d}$, $j=1,\ldots,h_i$, are $h_i$ distinct points, all contained in $M_{d}(S_i)$.

\subsection{}\label{sec:explainGT}  We conclude this section by a few additional remarks.
\smallskip

-- The kind of finiteness we have discussed for K3 surfaces, hyperk\"ahler manifolds, and abelian varieties
does not generalize to arbitrary varieties. In fact, it already fails for blow-ups of K3 surfaces. For a concrete
example, consider an automorphism $f\colon S\congpf S$ of infinite order of a K3 surface $S$.
Then consider the family $\pi\colon\kx\coloneqq {\rm Bl}_\Delta(S\times S)\stackrel{\sigma}{\to}
S\times S\stackrel{p_1}{\to}S$, which over a point $s\in S$ is the blow-up of $S$ in $s$.
Fix a sufficiently ample line bundle $L$ on $S$ and the induced $\pi$-ample line bundle
$p_1^*L(-E)$, where $E\to\Delta$ is the exceptional divisor.
Then for an infinite orbit $\{s_i\coloneqq f^{i}(s)\}$ the infinitely many fibres $\kx_{s_i}$ are all isomorphic.
These isomorphisms will be mostly unpolarized, as $\{f^{i\ast}L\}$ will be infinite for every ample line bundle $L$ on $S$.
Hence, in the corresponding moduli space, the orbit $M_{d-1}({\rm Bl}_s(S))$ will be infinite.


\smallskip

-- There are indeed instances where finiteness can be deduced by combining a local argument, showing discreteness of a
set, with algebraicity. For example, for a polarized K3 surface $(S,L)$ the quasi-projectivity of the Hilbert scheme ${\rm Hilb}^{P}(S\times S)$ 
(of closed subschemes of $S\times S$ with fixed Hilbert polynomial $P(n)=\chi(S,L^{2n})$)  implies
that the group $\Aut(S,L)$ of automorphisms $f\colon S\congpf S$ with $f^*L\cong L$ is finite, cf.\ \cite[Ch.\ 5]{HuyK3}.
Note that the finiteness of the group of polarized automorphisms implies that the moduli stacks discussed previously are
Deligne--Mumford stacks.

\smallskip

\begin{comment}
-- One may wonder whether there is any relation between the number of orbits of
the $\Aut(S_0)$-action on the set of ample line bundles $L$ of given square $(L)^2=d$
and the order of the stabilizer $\Aut(S_0,L_0)$ of any particular such polarization.
In other words, are the subsets $\{(S,L)\mid |M_d(S)|=\text{ const}\}\subset M_d$
and $\{ (S,L)\mid |\Aut(S,L)|=\text{ const}\}\subset M_d$ in any way related?
They are not, as the number $|\Aut(S,L)|$ is known to be bounded for $(S,L)\in M_d$. Indeed,
classified by Mukai's result. Improvements due to Kond\={o} eventually show $|\Aut_s(S,L)|\leq 3840$,
see \cite[Ch.\ 15]{HuyK3} for references. The non-symplectic part is also universally bounded by $66$,
cf.\ \cite[Cor.\ 15.1.14]{HuyK3}.
That there exists a bound at all can be deduced from a standard Hilbert scheme argument, so
eventually relying on the quasi-projectivity of certain moduli spaces.
\end{comment}


\section{Finiteness of hyperk\"ahler metrics and in twistor families}

Let $X$ be a compact hyperk\"ahler manifold, for example a K3 surface $S$, and $\omega\in H^{1,1}(X)$ a K\"ahler
class. Then there exists a unique hyperk\"ahler metric $g$ on $X$ whose K\"ahler form $g(I~,~)$
represents $\omega$. This can in particular be applied to the K\"ahler class provided by the first Chern class of an ample
line bundle $L$ on $X$. The associated hyperk\"ahler metric shall be denoted $g_L$ and then
$(X,g_L)$ is the underlying Riemannian manifold (with the complex structure of $X$ dropped).

We shall first address the question how many polarized hyperk\"ahler manifolds
$(X,L)\in M_d$ of fixed degree realize the same Riemannian manifold, i.e.\ 
for a given $(X_0,L_0)$ we study the set  $$M_d(X_0,g_{L_0})\subset M_d$$  of all 
$(X,L)\in M_d$ such that there exists an isometry  $(X_0,g_{L_0})\cong (X,g_L)$ 
between the  underlying Riemannian manifolds.

We will then turn to  families `orthogonal' to the moduli spaces $M_d$ provided by
the twistor construction. Recall that to each hyperk\"ahler metric $g$ there exists
an $S^2$ of complex structures compatible with $g$.
This leads to the twistor family $\pi\colon\kx\to\PP^1$ consisting of a complex manifold
$\kx$ with underlying differentiable manifold  $X\times \PP^1$ and the holomorphic
projection $\pi$ to the second factor. Each fibre $\kx_t$, $t\in\PP^1$, comes with an associated K\"ahler
class $\omega_t\in H^{1,1}(\kx_t)$. Altogether they span a positive three-space
$\langle\omega_t\rangle_{t\in\PP^1}\subset H^2(X,\RR)$, which is alternatively described
as $\RR\cdot{\rm Re}(\sigma)\oplus\RR\cdot{\rm Im}(\sigma)\oplus\RR\cdot\omega$. Here, $0\ne\sigma\in H^{2,0}(X)$ is the unique
(up to scaling) holomorphic two-form. We shall denote by $\kx_0$ the fibre that corresponds to $X$
and so $\omega_0=\omega$. Note that the Beauvillle--Bogomolov form is constant on $\{\omega_t\}$
or, in other words, $\int\omega_t^{2n}\equiv{\rm const}$.

 We then ask whether there are fibres $\kx_t$, $t\in\PP^1$,
biholomorphic to $X$ or such that $(\kx_{t},\omega_{t})\cong(X,\omega)$ as K\"ahler manifolds? 
Are the following sets finite:
 $$\{t\in\PP^1\mid (\kx_t,\omega_t)\cong (X,\omega)\}\text{~~ and ~~} \{t\in\PP^1\mid \kx_t\cong X\}?$$ We  prove finiteness in
the first case and for K3 surfaces with CM in the second.

\subsection{}  Consider the twistor family $\kx\to\PP^1$ associated with a 
compact hyperk\"ahler manifold $X$ endowed with a K\"ahler class.

\begin{lem}\label{lem:key}  There exist only finitely many $t\in \PP^1$ such that
the natural K\"ahler class $\omega_t$ is integral, i.e.\ given by an ample line bundle $L_t$ on $\kx_t$.
In particular, if $\omega$ is the class of an ample line bundle $L$, then at most
finitely many polarized fibres $(\kx_t,\omega_t)$ are isomorphic to $(X,L)$.
\end{lem}

\begin{proof}
The positive three-space $\langle\omega_t\rangle=(H^{2,0}\oplus H^{0,2})(X,\RR)\oplus\RR\cdot\omega$
intersects the lattice $H^2(X,\ZZ)$ in a positive definite lattice (of rank $\leq 3$). As the number of classes
of fixed square in a positive definite lattice is finite, the set
$\{\omega_t\}\cap H^2(X,\ZZ)$ is finite. As the class $\omega_t$ determines $t\in\PP^1$
up to complex conjugation, for only finitely many fibres $\kx_t$ the class $\omega_t$ can be integral.
\end{proof}

\begin{remark}\label{rem:Lagr}
Note that in most cases  the natural class $\omega_t$ on $\kx_t$  will be
integral or  just rational for only two of the fibres. Indeed, $\omega_t\in P\coloneqq
\RR\cdot{\rm Re}(\sigma)\oplus\RR\cdot{\rm Im}(\sigma)\oplus\RR\cdot\omega$
and  $P\cap H^2(X,\QQ)=\QQ\cdot\omega$ for very general
$(X,\omega\coloneqq{\rm c}_1(L))$ and then only the fibres corresponding to $X$ and its complex conjugate 
(corresponding to the positive plane $(H^{2,0}\oplus H^{0,2})(X,\RR)$ with reversed orientation) have rational $\omega_t$. 
At the other extreme, $P$ is defined over $\QQ$ if and only if $\rho(X)=h^{1,1}(X)$.

The remaining cases with $P\cap H^2(X,\QQ)$ of dimension two
are parametrized by a countable union of real Lagrangians in $M_d$.
Let us spell this out for K3 surfaces. Define $\Lambda_d\coloneqq \ell^\perp\subset\Lambda$ as
the orthogonal complement of a primitive class $\ell$ in the K3 lattice $\Lambda$ with $(\ell)^2=d$. Then
$M_d$ is an open subset of the arithmetic quotient $\tilde\OO(\Lambda_d)\setminus D$
of the period domain $D\subset\PP(\Lambda_d\otimes\CC)$ (viewed as the set of 
positive oriented planes in $\Lambda_d\otimes\RR$) by the stabilizer of $\ell$, cf.\ \cite[Ch.\ 6]{HuyK3}. Now,
for any $\alpha\in\Lambda_d$ with $(\alpha)^2>0$ consider the positive cone
$\kc_{\alpha^\perp}\subset\alpha^\perp\otimes\RR$. Note that $\alpha^\perp$ is of signature $(1,19)$.
The image $\kl_\alpha\subset M_d$ of the natural map
$$\kc_{\alpha^\perp}/\RR^*\,\hookrightarrow D\to \tilde\OO(\Lambda_d)\setminus D,$$
that sends $\beta\in\kc_{\alpha^\perp}$ to the plane spanned by $\alpha$ and $\beta$,
or rather its intersection with the open set $M_d$, describes the set of all polarized
K3 surfaces $(S,L)$ with $\alpha\in (H^{2,0}\oplus H^{0,2})(S,\ZZ)$. Each of the countably many $\kl_\alpha\subset M_d$ is of real dimension $19$ and isotropic with respect to the natural symplectic form on $M_d$. Compare this to \cite[Rem.\ 5.7]{BL}.
\end{remark}

Let us rephrase Lemma \ref{lem:key} more algebraically in the case of K3 surfaces.
Consider the moduli space $M_d$ of polarized K3 surfaces  $(S,L)$
of degree $d$. For $(S_0,L_0)\in M_d$ let $M_d(S_0,g_{L_0})\subset M_d$ be the set of polarized K3 surfaces $(S,L)$ such that the underlying Riemannian manifold $(S,g_L)$ is isometric to $(S_0,g_{L_0})$.

\begin{cor}\label{cor:Omet}
The set $M_d(S_0,g_{L_0})\subset M_d$ is always finite.
For a very general $(S_0.L_0)$ the set $M_d(S_0,g_{L_0})$
consists of $(S_0,L_0)$ and its conjugate $(\bar S_0,L_o^*)$.
However, there exist polarized K3 surfaces $(S_0,L_0)\in M_d$
with $|M_d(S_0,g_{L_0})|>2$.
\end{cor}

\begin{proof}
Indeed, $(S,L)\in M_d(S_0,g_{L_0})$ if and only if $S$ and $S_0$ with the K\"ahler
classes induced by $L$ and $L_0$ are isomorphic to fibres $(\ks_t,\omega_t)$
and $(\ks_0,\omega_0)$ of one twistor family $\ks\to\PP^1$. However, as explained above,
for only finitely many fibres of the twistor family associated with $(S_0,\omega_0)$  the K\"ahler class $\omega_t$ can be integral.
In fact, as discussed in Remark \ref{rem:Lagr}, for $(S_0,L_0)$ in the complement of the countable union
$\bigcup\kl_\alpha\subset M_d$ of real Lagrangians, only for the fibres $\ks_t$ corresponding to $S_0$ and to its conjugate $\bar S_0$ the
class $\omega_t$ will be integral and thus correspond to the K\"ahler class of the form $g_L$.

To construct examples with more interesting $M_d(S_0,g_{L_0})\subset M_d$, take $(S_0,L_0)$ with  $(L_0)^2=2$,
$T(S_0)=\left(\begin{matrix}2&1\\1&2\end{matrix}\right)$, and such that
$T(S_0)\oplus \ZZ\cdot L\subset H^2(S_0,\ZZ)$ is primitive. Then the fibres $(\ks_t,\omega_t)$  of the form $(S,L)\in M_2$ correspond to
elements $\alpha\in T(S_0)\oplus\ZZ(2)$ with $(\alpha)^2=2$. For example, $S_0$ corresponds to the basis vector $e_3$ and another
fibre $\ks_t$ corresponds to the standard vector $e_1$. As their orthogonal complements $T(S_0)$ and
$T(\ks_t)\cong\ZZ\cdot (e_1-2e_2)\oplus\ZZ\cdot e_3$ have distinct discriminants, the two fibres are not isomorphic or complex conjugate
to each other. Hence, in this case $|M_d(S_0,g_{L_0})|>2$.
\end{proof}

It should be possible, using the above construction, to show that $|M_d(S_0,g_{L_0})|$
is unbounded for varying $(S_0,L_0)\in M_d$ and fixed $d$, analogously to Section \ref{sec:unbound}.

\begin{remark}
The two sets $M_d(S_0),M_d(S_0,g_{L_0})\subset M_d$ associated with a polarized K3
surface $(S_0,L_0)\in M_d$ are not related and, in particular, not contained in each other.
Indeed, $|M_d(S_0,g_{L_0})|\leq2$ for all K3 surfaces $S_0$ with $(H^{2,0}\oplus H^{0,2})(S_0,\QQ)=0$.
However,  the latter condition is unrelated to the question how many non-isomorphic polarizations $L$ on
$S_0$ there are with $(L)^2=(L_0)^2$.
Also, the $M_d(S_0)$ come in algebraic families parametrized by quasi-projective varieties, see Corollary \ref{cor:parOS},
 whereas the $M_d(S_0,g_{L_0})$ come in families parametrized by the Lagrangians $\kl_\alpha$ discussed above.
\end{remark}
\subsection{}\label{sec:twistorfini} The lattice theory in the unpolarized situation is more involved and we will restrict for simplicity again
to the case of K3 surfaces. 

Let $T$ be a lattice of signature $(2,n-2)$ with a fixed basis $\gamma_1,\ldots,\gamma_n\in T$.
We consider Hodge structures of K3 type on $T$. Up to scaling, such a Hodge structure is given by a class
$\sigma\in T\otimes\CC$ with $(\sigma)^2=0$ and $(\sigma.\bar\sigma)>0$. We let
$\sigma_i\coloneqq (\gamma_i.\sigma)$, where we may choose $\sigma$ such that $\sigma_1=1$ (after permuting the $\gamma_i$ if necessary or
by assuming $(\gamma_1)^2>0$ from the start).
The \emph{period field} of $\sigma$ is defined as $$K_\sigma\coloneqq\QQ(\sigma_i)\subset\CC,$$
which is also generated by the coordinates of $\sigma$. Alternatively, $K_\sigma=\QQ(\mu_i)$, where $\sigma=\sum\mu_i\cdot\gamma_i$. 

The Hodge structure determined by $\sigma$
 is \emph{general} if there does not exist
a proper primitive sublattice $T'\subset T$ with $\sigma\in T'\otimes\CC$.

Consider now an isometric embedding $\varphi\colon T\,\hookrightarrow T\oplus\ZZ\cdot e$ with $(e)^2=d>0$ and such
that 
\begin{equation}\label{eqn:varphi}
\varphi_\CC(\sigma)=\lambda\cdot \sigma+\lambda'\cdot\bar\sigma+\nu\cdot e.
\end{equation} 

\begin{lem}\label{lem:CMlambda}
Assume that $\sigma\in T\otimes\CC$ defines a general Hodge structure of K3 type on $T$ such that
$K_\sigma$ is a subfield of a CM field $K$. Then there exist at most finitely many isometric embeddings $\varphi\colon T\,\hookrightarrow T\oplus\ZZ\cdot e$ satisfying (\ref{eqn:varphi}).
\end{lem}

\begin{proof}
Consider an isometric embedding $\varphi$ satisfying (\ref{eqn:varphi}). We think of $\varphi$ in
terms of the integral matrix $(a_{ij}| b_j)$, where $\varphi(\gamma_i)=\sum a_{ij}\cdot \gamma_j+b_i\cdot e$ with $\gamma_i$ a basis
of $T(S)$ as above. Then $\varphi_\CC(\sigma)$ corresponds to $(a_{ij}|b_j)\cdot (\mu_i)$ and (\ref{eqn:varphi}) becomes the system of equations
$\sum a_{ij}\mu_j=\lambda\mu_i+\lambda'\bar\mu_i$, $i=1,\ldots,n$, and $\sum b_j\mu_j=\nu$. This shows $\lambda,\lambda',\nu\in K$.
Indeed, $\lambda,\lambda'$ can be written as a rational linear combinations of the $\mu_i,\bar\mu_i$ unless $\det\left(\begin{matrix}\mu_i&\bar\mu_i\\
\mu_j&\bar\mu_j\end{matrix}\right)=0$, i.e.\  $\mu_i\bar\mu_j\in\RR$, for all $i\ne j$. But then $1=\sigma_1=\sum(\mu_i\gamma_i.\gamma_1)$
implies $\bar\mu_j=\sum\mu_i\bar\mu_j(\gamma_i.\gamma_1)\in\RR$ for all $j$, which would yield the contradiction $\sigma=\bar\sigma$.
Eventually use that $K_\sigma=\QQ(\mu_i)$ is contained in the CM field $K$ which is closed under complex conjugation.

Next observe that, as $a_{ij},b_j\in\ZZ$, there exists an $N\in \ZZ$ independent of $\varphi$ such that $N\lambda,N\lambda',N\nu\in\ko_{K}$.
As $\varphi$ is an isometry, one also has $(\sigma.\bar\sigma)=(\varphi(\sigma).\varphi(\bar\sigma))$, which translates into
\begin{equation}\label{eqn:varphi2}
|\lambda|^2+|\lambda'|^2+|\nu|^2(d/(\sigma.\bar\sigma))=1.
\end{equation}


As any embedding $g\colon K\,\hookrightarrow \CC$ commutes with complex conjugation, 
$g$ applied to (\ref{eqn:varphi2}) also shows $|g(\lambda)|^2+|g(\lambda')|^2+|g(\nu)|^2(d/g(\sigma.\bar\sigma))=1$.
Observe that $g(\sigma.\bar\sigma)>0$ and that, therefore, the last summand is non-negative. Indeed, choose $z\in \CC$ such that $(z\sigma,\bar z\bar\sigma)=1$. Then
$|g(z)|^2\,g(\sigma.\bar\sigma)=(g(z\sigma),g(\bar z\bar \sigma))=1$, as $g$ commutes with complex conjugation.
Hence, for $N\lambda,N\lambda'\in\ko_K$ one has $|g(N\lambda)|\leq N$ and $|g(N\lambda')|\leq N$ for all $g\colon K\,\hookrightarrow\CC$.
By Minkowski theory there are only finitely many such $N\lambda,N\lambda'\in{\mathcal O}_{K}$. 

As there are only finitely many possibilities for  $\lambda$ and $\lambda'$, it suffices to show that they determine $\varphi$ essentially
uniquely. Indeed, if $\varphi$ and $\varphi'$ are both isometric embeddings satisfying (\ref{eqn:varphi}) with the same $\lambda,\lambda'$, then
$\sigma\in{\rm Ker}(\psi-\psi')\otimes\CC$, where $\psi,\psi'\colon T\to T$ are the compositions of $\varphi$, $\varphi'$ with the projection to $T$. Hence, using the assumption that the Hodge structure is general, one finds
$\psi=\psi'$. Using that $\varphi$ and $\varphi'$ are both isometric embeddings allows one to conclude.
\end{proof}

\begin{remark}\label{rem:finiteunity}
The proof shows more. It allows one to control the number of  isometric embeddings $\varphi\colon T\,\hookrightarrow T\oplus\ZZ\cdot e$ satisfying (\ref{eqn:varphi}). Indeed, up to a certain sign, such a $\varphi$ is determined by a root of unity in $\ko_{K_\sigma}$. Hence, the number of such maps $\varphi$
is  bounded by twice the number of roots of unity in ${\mathcal O}_{K_\sigma}$, which is bounded from above by a number
depending only on $[K_\sigma:\QQ]$. 
\end{remark}

\begin{remark}\label{rem:relax}
The hypotheses can be relaxed a little. For example, one can consider isometric embeddings of $T$ into a fixed finite index
overlattice of $T\oplus\ZZ\cdot e$ (e.g.\ the saturation of $T(S)\oplus\ZZ\cdot e$ in $H^2(S,\ZZ)$). Indeed, instead of working with $\varphi$ in the proof above one uses
$n\cdot \varphi$, where $n$ is the index of the overlattice, and then replaces the left hand side of (\ref{eqn:varphi2}) by $n^2$. 
 Also, $(e)^2\in2\ZZ_{>0}$ can be replaced by $(e)^2\in K_\sigma$ satisfying $g((e)^2)\in\RR_{>0}$ for
 all embeddings $g\colon K_\sigma\,\hookrightarrow\CC$.
\end{remark}

Consider the twistor space ${\mathcal S}\to\PP^1$ associated with a polarized K3 surface $(S_0,L_0)$.

\begin{prop}\label{prop:finitetwistor}\label{prop:finTwistor}
Assume that $S_0$ is a K3 surface with CM. Then there exist at most finitely many $t\in\PP^1$ such that the fibre
${\mathcal S}_t$ is isomorphic to $S_0$. In fact, there are at most finitely many $t\in\PP^1$ 
such that $\ks_t$ and $S_0$ are Fourier--Mukai partners.
\end{prop}

The arguments below only use that the period field $K_\sigma$ is a CM field and, therefore,
in particular algebraic. It is known, that under the additional assumption that $S_0$ is defined
over $\bar\QQ$  the period field  $K_\sigma$ is algebraic if and only if $S_0$ has CM, see \cite{Tret}.

\begin{proof}
Recall that a K3 surface with CM is a projective K3 surface $S_0$ for which the endomorphism field
$K={\rm End}_{\rm Hdg}(T(S_0)\otimes\QQ)$ of endomorphisms of the rational Hodge structure $T(S_0)\otimes\QQ$
given by the transcendental lattice $T(S_0)$ is a CM field with $\dim_K (T(S_0)\otimes\QQ)=1$, see \cite[Ch.\ 3]{HuyK3} for references.
It is known that any K3 surface with CM is defined over $\bar\QQ$, but this will not be used in the argument.

Let $K_\sigma$ be the field generated by the periods $\sigma_i$ of $S_0$ as above.  We claim that $K_\sigma\subset K$. For this it is enough to show
that $(\gamma.\sigma)\in K$ for all $\gamma\in T(S_0)$. As $\dim_K(T(S_0)\otimes\QQ)=1$ by assumption, $T(S_0)$ is
contained in the $\QQ$-span of all $\alpha(\gamma_1)$, $\alpha\in K\subset{\rm End}(T(S_0)\otimes\QQ)$, where $\gamma_1$ is
as above the first vector of a basis $\gamma_i\in T(S_0)$ with $(\sigma.\gamma_1)=1$.
Hence, for $\gamma\in T(S_0)$, $(\sigma.\gamma)$ is a $\QQ$-linear combination of numbers 
of the form $(\sigma.\alpha(\gamma_1))$, $\alpha\in K$. As $(\sigma.\alpha(\gamma_1))=(\alpha'\sigma.\gamma_1)=\alpha'\in K$, this proves $K_\sigma\subset K$. (Recall that $\alpha\mapsto\alpha'$ is an automorphism of the CM field $K$ which under any embedding corresponds to complex conjugation.)

Let us now turn to the proposition itself. Clearly, it is enough to show the finiteness of Fourier--Mukai partners in a twistor family.
Suppose that a twistor fibre $\ks_t$ is derived equivalent to $S_0$, i.e.\ that there exists an exact linear equivalence $\Db(\ks_t)\cong\Db(S_0)$ between their derived categories.
Hence, in addition to the identification of lattices $H^2(\ks_t,\ZZ)=H^2(S_0,\ZZ)$
induced by the twistor diffeomorphism $\ks\cong S_0\times\PP^1$, any chosen
equivalence $\Db(S_0)\congpf\Db(\ks_t)$ provides us with an additional Hodge isometry 
between the transcendental lattices $T(S_0)\congpf T(\ks_t)$. The composition of the two
induces an isometric embedding $\varphi\colon T(S_0)\congpf T(\ks_t)\,\hookrightarrow  H^2(S_0,\ZZ)$ with the additional property
that $\varphi(\sigma)$ is contained in $\CC\cdot\sigma\oplus\CC\cdot\bar\sigma\oplus\CC\cdot e$, where $e\coloneqq {\rm c}_1(L_0)$. Hence, $\varphi(T(S_0))$ is contained in the saturation
of $T(S_0)\oplus\ZZ\cdot e\subset H^2(S_0,\ZZ)$. For simplicity assume that $\varphi\colon T(S_0)\,\hookrightarrow T(S_0)\oplus\ZZ\cdot e$,
but see Remark \ref{rem:relax}.

According to Lemma \ref{lem:CMlambda}, under our assumptions there exist only finitely many such $\varphi$. Here we use that
the transcendental lattice is general, for it is the minimal primitive sublattice containing $\sigma$ in its complexification.
As $\varphi(\sigma)$ determines
$t$ up to sign, only finitely many fibres $\ks_t$ are derived equivalent to $S_0$.
\end{proof}

\begin{remark}\label{rem:132}
Using Remark \ref{rem:finiteunity}, we conclude that for a polarized K3 surface $(S_0,L_0)$ with CM the cardinality of the two finite sets
\begin{equation}\label{eqn:boundnumber}
|\{t\in\PP^1\mid{\mathcal S}_t\cong S_0\}|\text{  and }|\{t\in\PP^1\mid\Db({\mathcal S}_t)\cong\Db(S_0)\}|
\end{equation}
can be bounded by a constant $c(K)$ only depending on the CM field $K={\rm End}_{\rm Hdg}(T(S_0)\otimes\QQ)$.
In fact, as $[K:\QQ]=\rk\, T(S_0)\leq 21$, the Euler function $\varphi(m)$  of the
$m$-th roots of unity that can occur in the proof of Lemma \ref{lem:CMlambda} is bounded by $21$ and hence $m\leq 66$.
Taking complex conjugation into account, this shows that the numbers in (\ref{eqn:boundnumber}) are universally bounded by $132$.
\end{remark}

Note however that infinitely many of the fibres $\ks_t$ may come with a polarization
$L_t$ yielding infinitely many points $(\ks_t,L_t)$ in the moduli space of
polarized K3 surfaces $M_d$ of fixed degree, which could be loosely phrased as saying that
a twistor line usually intersects the quasi-projective moduli space $M_d$ in infinitely many points, only
that the underlying K3 surfaces will not be isomorphic to each other.

\subsection{}
It may be instructive to look at K3 surfaces $S_0$ of maximal Picard number $\rho(S_0)=20$. Those are
known to have CM, cf.\ \cite[Rem.\ 3.3.10]{HuyK3}. In this case, the arguments simplify. Indeed, the transcendental lattice $T(S_0)$ is then positive definite (of rank two) 
and  there exist only finitely many isometric embeddings of $T(S_0)$ into any other fixed positive definite lattice, e.g.\ $T(S_0)\oplus\ZZ\cdot e$. So finiteness in Lemma \ref{lem:CMlambda} follows directly. 

\smallskip

K3 surfaces of maximal Picard rank can also be used to construct  examples of twistor families
with isomorphic distinct fibres. Indeed, if $S_0$ is a K3 surface with $T(S_0)\cong\ZZ(d)^{\oplus 2}$ with
orthogonal basis $e_1,e_2$ and such that there exists a polarization $L$ of degree $d$ for which
$T(S_0)\oplus\ZZ\cdot L\subset H^2(S_0,\ZZ)$ is saturated, then the two fibres $\ks_1,\ks_2$ corresponding to $T(\ks_i)=e_i^\perp$ have  both transcendental lattices isomorphic
to $T(S_0)$. Therefore, by Orlov's result, $\Db(\ks_i)\cong\Db(S_0)$ and, as $\rho(\ks_i)=20$, in fact $\ks_i\cong S_0$. 

\smallskip

A closer inspection of this case also reveals that Proposition \ref{prop:finitetwistor} will be difficult to strengthen. For example, one could ask how many of the  fibres $\ks_t$  are isogenous to $S_0$, i.e.\ such that there exists a Hodge isometry
$H^2(S,\QQ)\cong H^2(\ks_t,\QQ)$ (or, equivalently, a Hodge isometry $T(S)\otimes\QQ\cong T(\ks_t)\otimes\QQ$), or 
how many of them have the same Chow motive $\hh(S_0)\cong \hh(\ks_t)$, see \cite{HuyMot} for the relation between the two notions. However, finiteness fails in these settings.
Indeed, for $\rho(S_0)=20$ the projective fibres $\ks_t$ are up to conjugation uniquely determined
by primitive classes $e_t$ in (the saturation of) $T(S_0)\oplus\ZZ\cdot e\subset H^2(S_0,\ZZ)$. It is essentially
$(e_{t_1})^2/(e_{t_2})^2\in\QQ^\ast/\QQ^{\ast2}$ that decides whether $\ks_{t_1}$ and $\ks_{t_2}$ are isogenous.
So one will usually have infinitely many fibres $\ks_t$ that are isogenous to $S_0$ and infinitely many that are not.

\begin{comment}
\begin{remark}
There is yet another question related to the above, especially to
Theorem \ref{thm:HK} and Corollary \ref{cor:Omet}. For given $X_0$ how
big is the set of $(X,L)\in M_d$ for which the underlying $X$ is isomorphic to
a twistor fibre $\kx_t$ of a twistor family $\kx\to\PP^1$ associated with $X_0$ endowed
with an arbitrary K\"ahler class (and not necessarily the one associated with $L$)?
This set is always infinite, which is not difficult to prove and certainly expected,
as connected components of the moduli space of compact hyperk\"ahler manifolds 
are twistor connected.
\end{remark}
\end{comment}

\section{Finiteness via Kuga--Satake and abelian varieties}\label{sec:KS}

An approach to the finiteness of polarizations of fixed degree on an abelian variety $A_0$ modulo
the action of $\Aut(A_0)$ similar to the one presented in Section \ref{sec:FinHK}
can be worked out. It provides a new proof of the classical result of 
Narasimhan and Nori \cite{NarNor}, cf.\ \cite[Sect.\ 18]{Milne}. Note that the finiteness of $\ka_{g,d}(A_0)\subset \ka_{g,d}$
can be quickly reduced to the case of principally polarized abelian varieties via Zarhin's trick
which provides a (non-canonical) quasi-finite morphism $\ka_{g,d}\to\ka_{8g}\coloneqq\ka_{8g,1}$, $(A,L)\mapsto ((A\times \hat A)^4,\kl)$ to the moduli space of principally polarized abelian varieties, see e.g.\ the account
in \cite[Sect.\ 4.1]{OS}.

The question how many principal polarizations an abelian variety can admit
has been studied by Lange \cite{La}, who in particular
describes bounds for the cardinality of $\ka_g(A_0)\subset \ka_g$ for $A_0$ with real multiplication. Further results for
products of elliptic curves can be found in \cite{Hay,La2}, where, for example, it is shown that $|\ka_2(E_1\times E_2)|$ is unbounded
for isogenous elliptic curves $E_1,E_2$ without complex multiplication. See also \cite{Howe1,Howe2}.

\medskip

We shall now indicate an alternative proof of Theorem \ref{thm:HK} based on the Kuga--Satake construction which
reduces the problem to the finiteness for abelian varieties. For simplicity, we  restrict to the case of K3 surfaces and leave the necessary
modifications in the higher-dimensional case to the reader.
 
Starting with the Hodge structure of weight two  $H^2(S,\ZZ)_{L\text{-pr}}$ of an arbitrary polarized K3 surface $(S,L)$,
the Kuga--Satake construction produces an abelian variety ${\rm KS}(S,L)$ of dimension $g=2^{19}$. By choosing a pair of positive orthogonal vectors in $H^2(S,\ZZ)_{L\text{-pr}}$,
e.g.\ $e_1+f_1, e_2+f_2$ in the two copies of the hyperbolic plane $U$ in the decomposition
$ H^2(S,\ZZ)_{L\text{-pr}}\cong E_8(-1)^{\oplus 2}\oplus U^{\oplus 2}\oplus\ZZ(-d)$,
one can define a polarization on ${\rm KS}(S,L)$ of a fixed degree $d'$, cf.\ \cite[Ch.\ 4]{HuyK3}. Suppressing
the choice of finite level structures, this leads to a morphism
\begin{equation}\label{eqn:KS} M_d\to \ka_{g,d'},
\end{equation}
which is known to be quasi-finite, see  \cite{Andre,Rizov} and \cite[Prop.\ 5.10]{Maulik}.
 
The Kuga--Satake construction applies to any Hodge structure of K3 type, in particular to the transcendental lattice $T(S)$ of a K3 surface, which
is independent of the polarization, and to $T(S)\oplus L^\perp$. Here, $L^\perp\subset\NS(S)$ denotes the primitive sublattice
orthogonal to an ample $L\in \NS(S)$. Moreover, $${\rm KS}(T(S)\oplus L^\perp)\cong {\rm KS}(T(S))^{2^{\rho(S)-1}}.$$
In particular, for two different polarizations $L_1,L_2$ on a K3 surface $S$ one has ${\rm KS}(T(S)\oplus L_1^\perp)\cong {\rm KS}(T(S))^{2^{\rho(S)-1}}\cong{\rm KS}(T(S)\oplus L_2^\perp)$.
Now, the natural inclusion $\iota_i\colon T(S)\oplus L_i^\perp\subset H^2(S,\ZZ)_{L_i\text{-pr}}$ defines a finite index sublattice, leading to an isogeny $${\rm KS}(T(S))^{2^{\rho(S)-1}}\cong{\rm KS}(T(S)\oplus L^\perp)\twoheadrightarrow {\rm KS}(S,L).$$ Its degree  depends only
on the index of $\iota_i$. There may be infinitely many polarizations $L_i$ of fixed degree but only finitely many
inequivalent ones under the action of $\OO(\NS(S))$. Hence,
the inclusions $\iota_i$ have bounded index and, therefore, the ${\rm KS}(S,L_i)$ are all
quotients of bounded degree of the abelian variety ${\rm KS}(T(S))^{2^{\rho(S)-1}}$.
 
This then shows the following result which allows one to deduce from the finiteness of $\ka_{g,d}(A_0)$ for abelian varieties \cite{NarNor,Milne} the
finiteness of $M_d(S_0)$ for K3 surfaces (and more generally of $M_d(X_0)$ for compact hyperk\"ahler manifolds).

\begin{cor}
For any K3 surface $S_0$ the image of $M_d(S_0)\subset M_d$ under the quasi-finite map (\ref{eqn:KS})
is contained in a finite union of sets $\ka_{g,d'}(A_i)\subset \ka_{g,d'}$, where the finitely many abelian varieties
$A_i$ are quotients of the abelian variety ${\rm KS}(T(S))^{2^{\rho(S)-1}}$ of a fixed degree.\qed
\end{cor}

Using the finiteness of $\ka_{g,d'}(A_i)$ \cite{NarNor} and the finiteness of the fibres of (\ref{eqn:KS}), this 
proves Theorem \ref{thm:HK} once again.


\section{Parametrization by special cycles and FM interpretation}\label{sec:paraFM}

The discussion in Section \ref{sec:FinHK} can be understood more explicitly in the case of K3 surfaces,
although most issues related to lattice theory remain, for example the decomposition of the moduli
space according to the various $\ell\in N$ and the possible choices for the
primitive embeddings $N\,\hookrightarrow \Lambda$ still occur (unless $N$ is of small rank as in \cite{Beauv}).

\subsection{} Moduli spaces of lattice polarized K3 surfaces have first been studied in \cite{Dolg} and later also in \cite{Beauv}.
Dolgachev shows (using the global Torelli theorem) that the moduli space of `ample $N$-polarized' K3 surfaces $M_N$ injects into the quotient
$\tilde \OO(T_i)\setminus D_{T_i}$. Here, $\tilde\OO(T_i)\subset\OO(T_i)$ is the finite index subgroup of orthogonal transformations that
extend to all of $\Lambda$ by the identity on $\eta_i(N)$. Note that in \cite{Dolg} the embedding $N\,\hookrightarrow \Lambda$
is actually fixed so that only one period domain $D_i$ has to be considered and that $\iota(N)$ is only requested to contain an ample class (but 
without fixing $\ell\in N$ or its image $\iota(\ell)$). So there exists a quasi-finite morphism $M_{N,d}\to M_N$ inducing injections
$M_{(N,\ell)}\,\hookrightarrow M_N$. In \cite{Beauv} the rank of $N$ is small enough
to ensure that the embedding is actually unique, so that again only one period domain occurs.

For K3 surfaces there exists an interpretation of all points in the fibre of
$\pi_i\colon M^i_{N,d}\to \OO(T_i)\setminus D_i$ through a very general
$(S_0,\iota_0)$. They correspond to $N$-polarized K3
surfaces whose transcendental part is Hodge isometric to $T_i\cong T(S_0)$. Those are known to be the Fourier--Mukai
partners of $S_0$ with  N\'eron--Severi lattice isomorphic to $\NS(S_0)$, see \cite[Ch.\ 16]{HuyK3} for references.

Hence, for K3 surfaces $\pi_i$ factorizes as $$\pi_i\colon M^i_{N,\ell_j}\,\hookrightarrow \tilde \OO(T_i)\setminus D_i\twoheadrightarrow \OO(T_i)\setminus
D_i,$$ where the degree of the second map is essentially the number of Fourier--Mukai partners for the very general K3 surfaces
parametrized by $M^i_{N,d}$ with fixed N\'eron--Severi lattice $\NS(S_0)$.

\subsection{}
As the sets $M_d(S_0)$ (or $\ka_{g,d}(A_0)$, $M_d(X_0)$, etc.) are finite, one may wonder whether they can be realized
as fibres of a finite map from $M_d$ to some variety or space. Clearly, since $|M_d(S_0)|$ for varying $(S_0,L_0)\in M_d$
is unbounded (cf.\ Section \ref{sec:unbound}), this cannot be true literally. However, we shall explain that this idea
can be turned into a correct statement, which then also sheds light on the distributions of the $M_d(S_0)$.
For this purpose it is more convenient to replace  the sets $M_d(S_0)\subset M_d$ by the sets
 $$M_d(\Db(S_0))\subset M_d$$ of all
$(S,L)\in M_d$ for which there exists an exact linear equivalence  
\begin{equation}\label{eqn:FM}
\Db(S_0)\cong \Db(S)
\end{equation}
between the bounded derived categories of coherent sheaves on $S$ and $S_0$.
According to a result of Mukai and Orlov, the condition is equivalent to the existence of  a Hodge isometry 
\begin{equation}\label{eqn:OT}
T(S)\cong T(S_0).
\end{equation}
Recall that for $\rho(S_0)\geq 12$ the existence of  a Hodge isometry (\ref{eqn:OT}) is in fact equivalent to the existence
of an isomorphism $S\cong S_0$. In general, as $M_d(S_0)$ also the set
 $M_d(\Db(S_0))$ is finite and the arguments in Section \ref{sec:FinHK} in fact reprove this result.
See \cite[Ch.\ 16]{HuyK3} for references and further details and \cite{HP} proving finiteness
of Fourier--Mukai orbits in the moduli space of quasi-polarized K3 surfaces.

\begin{cor}\label{cor:parOS}
Consider a polarized K3 surface $(S_0,L_0)\in M_d$ with transcendental lattice $T\coloneqq T(S_0)$ and its
associated period domain $D_T\subset\PP(T\otimes\CC)$. Then there exist 
a   quasi-projective variety $M_{S_0}$  of dimension $20-\rho(S_0)$ and morphisms
$$\xymatrix{M_{S_0}\ar[r]^-\Phi\ar[d]_\pi&M_d\\
~~~~~~~~~~\bar M_{S_0}\coloneqq \OO(T)\setminus D_T,&}$$
such  that
\begin{enumerate}
\item[(i)] $\pi$ is quasi-finite and dominant;
\item[(ii)] $M_d(\Db(S_0))=\Phi(\pi^{-1}(t_0))$ for some point $t_0\in \bar M_{S_0}$;
\item[(iii)] $M_d(\Db(S_t))=\Phi(\pi^{-1}(t))$ for very general $t\in\bar M_{S_0}$ and any $(S_t,L_t)\in\Phi(\pi^{-1}(\pi(t)))$.\qed
\end{enumerate}
\end{cor}

In other words, the Fourier--Mukai orbits $M_d(\Db(S_0))$ are indeed (images of) fibres of a certain quasi-finite morphism of some subvariety of $M_d$
of dimension $20-\rho(S_0)$.

\begin{remark} Note that for the case of Picard rank one, $T(S_0)\cong\Lambda_d$ and $\pi$ is simply the composition
$$M_d\,\hookrightarrow \tilde \OO(\Lambda_d)\setminus D_d\to \OO(\Lambda_d)\setminus D_d.$$
This map has been studied in \cite{Hosono,Og,Stel}. Its degree, which is the number of Fourier--Mukai partners for the very general K3 surface,
is known to be  $2^{\tau(d)-1}$, where $\tau(d)$ is the number of prime divisors of $d/2$.
\end{remark}


\end{document}